\documentclass[titlepage,12pt]{article}
\usepackage{amsmath,amsthm,amssymb}
\usepackage[mathscr]{eucal}
\usepackage[latin1]{inputenc}
\usepackage{graphicx}
\usepackage{rotating}
\usepackage{soul}
\newtheorem{Def}{Definition}
\newtheorem{The}{Theorem}

\newtheorem{Cor}[The]{Corollary}

\newtheorem{Exam}[The]{Example}
\theoremstyle{definition}
\newtheorem{Rem}[The]{Remark}
\numberwithin{equation}{section} \numberwithin{The}{section}

\usepackage{color}

\begin{document}

\title{Sufficient conditions for some transform orders based on the quantile density ratio}

\author{ Antonio Arriaza \\ Dpto. Estad\'{\i}stica e Investigaci\'{o}n Operativa \\ Universidad de C\'{a}diz \\ Facultad de Ciencias Econ\'{o}micas y Empresariales \\ 11002 C\'{a}diz, SPAIN, \\ e-mail: \texttt{antoniojesus.arriaza@uca.es} \\ F\'{e}lix Belzunce and  Carolina Mart\'{i}nez-Riquelme \\ Dpto.
Estad\'{\i}stica e Investigaci\'{o}n Operativa \\ Universidad de Murcia
\\ Facultad de Matem\'{a}ticas, Campus de Espinardo \\
30100 Espinardo (Murcia), SPAIN, \\ e-mail: \texttt{belzunce@um.es, carolina.martinez7@um.es} } 

\date{}

\maketitle

\begin{abstract} In this paper we focus on providing sufficient conditions for some transform orders for which the quantile densities ratio is non-monotone and, therefore, the convex transform order does not hold. These results are interesting for comparing random variables with a non-explicit expression of their quantile functions or they are computationally complex. In addition, the main results are applied to compare two Tukey generalized distributed random variables and to establish new relationships among non-monotone and positive aging notions.



\end{abstract}

\section{Introduction}

Among the different criteria to compare random variables in the context of stochastic orders, there exist a group of criteria known in the literature as  transform orders (see M\"{u}ller and Stoyan, 2002, Shaked and Shanthikumar, 2007 and Belzunce et. al, 2015, for a review and an introduction to the topic). These criteria  have been used in different contexts such as reliability, risk theory and economics, and they are related to the comparison of two random variables in terms of their concentration (see Arnold and Sarabia, 2018). 

The definition of these orders is based on the notion of the quantile function, that we provide next. Let us consider a random variable $X$, with distribution function $F$, then the quantile function of $X$ is given by
\[
F^{-1}(p)=\inf\{x\in \mathbb R | F(x)\ge p\}, \text{ for all }p\in (0,1).
\]

The quantile function  is defined as a generalized inverse of the distribution function and it is the probabilistic counterpart of the percentile notion in descriptive statistics. In this paper, we will consider non-negative random variables with finite expectation and interval support (bounded or unbounded), denoted by $(l_X,u_X)$ for the random variable $X$. In short, that means, random variables are assumed to have density functions. In this case, assuming differentiability, the derivative of the quantile function is given by 
\[
\frac{d}{d p}F^{-1}(p)=\frac{1}{f(F^{-1}(p))},
\]
where $f$ is the density function associated to $F$. This function is known as the quantile density or sparsity function (see Parzen, 1979). 

Now, we proceed to recall the definitions of the different transform orders. Historically, the earliest stochastic orders considered as transform orders are the convex transform and the star-shaped order. The convex order was introduced by van Zwet (1964) for the comparison of two non-negative random variables in terms of their skewness. The star-shaped order can be found in Marhsall and Olkin (1979).

\begin{Def}
Let us consider two non-negative random variables $X$ and $Y$, with distribution functions $F$ and $G$ and density functions $f$ and $g$, respectively. 

\begin{itemize}

\item[a)] It is said that $X$ is smaller than $Y$ in the \textbf{convex transform order}, denoted by $X\leq_{c}Y$, if
\[
\frac{f(F^{-1}(p))}{g(G^{-1}(p))} \text{ is increasing in }p\in(0,1).
\]

\item[b)] It is said that $X$ is smaller than $Y$ in the \textbf{star-shaped order}, denoted by $X\leq_{\star}Y$, if
\[
\frac{G^{-1}(p)}{F^{-1}(p)} \text{ is increasing in }p\in(0,1).
\]

\end{itemize}
\end{Def}

Therefore, the convex and the star-shaped orders are defined in terms of the monotonicity of the ratio of the quantile densities and the quantile functions, respectively. These criteria can also be defined in terms of the convexity and the star-shaped property of the function $G^{-1}(F(x))$, respectively. It is well-known that the transformation $G^{-1}(F(x))$ maps the random variable $X$ onto $Y$, that is, $Y=_{\textup{st}}G^{-1}(F(X))$. Therefore, as it was pointed out by van Zwet (1964), if  $X\leq_{c}Y$, ``apart from an overall linear change of scale, such a non-decreasing, convex transformation of a random variable effects a contraction of the lower part of the scale of measurement and an extension of the upper part... the transformation should produce what one intuitively feels to be an increased skewness to the right". On the other hand, according to Doksum (1969), if ``$X\leq_{\star}Y$ then one would expect $Y$ to be ``more skewed to the left" than $X$. Moreover, if $X$ and $Y$ are random variables that represent ``time to failure," then $X\leq_{\star}Y$ indicates that $X$ is more subject to wear out than $Y$." At this point, we would like to clarify that both authors refer to the same type of asymmetry in different terms. Furthermore, the convex order can be restated in terms of the hazard rate. In particular, if we denote $\overline F=1-F$ and $\overline G = 1- G$, these being the survival functions and denote $r=f/\overline F$ and $s=g/\overline G$, these being the hazard rate functions of $X$ and $Y$, respectively, then $X\leq_{c}Y$ if, and only if,
\[
\frac{r(F^{-1}(p))}{s(G^{-1}(p))} \text{ is increasing in }p\in(0,1).
\]

The function $r(F^{-1}(p))$ is also known as the hazard quantile function (see Nair et. al, 2013). 

Motivated from the previous characterization, Kochar and Wiens (1987) introduce two new stochastic orders based on the mean residual quantile function. Given a random variable $X$, with distribution function $F$, the mean residual life function is given by
\[
m(t)=E[X-t|X>t]=\frac{\displaystyle \int_t^{+\infty} \overline F(x)dx}{\overline F(t)},\text{ for all }t\text{ such that }\overline F(t)>0.
\]

\begin{Def}
Let us consider two non-negative random variables $X$ and $Y$, with distribution functions $F$ and $G$ and mean residual life functions $m$ and $l$, respectively. 
\begin{itemize}

\item[a)] It is said that $X$ is smaller than $Y$ in the \textbf{dmrl order}, denoted by $X\leq_{dmrl}Y$, if
\[
\frac{l(G^{-1}(p))}{m(F^{-1}(p))}=\frac{\displaystyle \int_{G^{-1}(p)}^{u_Y}\overline G(x)dx}{\displaystyle \int_{F^{-1}(p)}^{u_X}\overline F(x)dx} \text{ is increasing in } p\in(0,1),
\]
or, equivalently, if
\begin{equation}\label{def-dmrl}
\frac{\displaystyle \int_p^1 \frac{1-q}{g(G^{-1}(q))}dq}{\displaystyle \int_p^1 \frac{1-q}{f(F^{-1}(q))}dq} \text{ is increasing in } p\in(0,1).
\end{equation}
\item[b)] It is said that $X$ is smaller than $Y$ in the \textbf{nbue order}, denoted by $X\leq_{nbue}Y$, if
\[
\frac{E[Y]}{E[X]}\le \frac{l(G^{-1}(p))}{m(F^{-1}(p))}=\frac{\displaystyle \int_{G^{-1}(p)}^{u_Y}\overline G(x)dx}{\displaystyle \int_{F^{-1}(p)}^{u_X}\overline F(x)dx}, \text{ for all } p\in(0,1). 
\]

\end{itemize}
\end{Def}

The function $m(F^{-1}(p))$ is also known as the mean residual quantile function (see Nair et. al, 2013). From the meaning of this function in terms of reliability, the dmrl and nbue orders can be interpreted in the sense that $X$ is more subject to wear out than $Y$, as in the case of the star-shaped order.

Recently,  Arriaza et. al (2017) introduced a new transform order, the qmit order, based on the mean inactivity time function. This order occupies an intermediate position between the convex transform order and the starshaped order. Hence, it can be interpreted in terms of skewness as well. On the other hand, the qmit order compares two random variables in terms of the quotient of their mean inactivity time at any quantile and it can be considered as a dual order of the dmrl order. Furthermore, the qmit order is coherent with the interpretation of the convex transform and starshaped orders in terms of reliability. Hence, if $X\leq_{qmit} Y$ one would expect that $Y$ represents the lifetime of a more reliable system.

Given a non-negative random variable $X$, with distribution function $F$, the mean inactivity time function is given by
\[
\overline m(t)=E[t-X|X<t]=\frac{\displaystyle \int_{0}^t F(x)dx}{ F(t)},\text{ for all }t\text{ such that }F(t)>0.
\]

\begin{Def}
Let us consider two non-negative random variables $X$ and $Y$, with distribution functions $F$ and $G$ and mean inactivity time  functions $\overline m$ and $\overline l$, respectively. We say that $X$ is smaller than $Y$ in the \textbf{qmit order}, denoted by $X\leq_{qmit}Y$, if
\[
\frac{\overline l(G^{-1}(p))}{\overline m(F^{-1}(p))}=\frac{\displaystyle \int_{l_Y}^{G^{-1}(p)} G(x)dx}{\displaystyle \int_{l_X}^{F^{-1}(p)} F(x)dx} \text{ is increasing in } p\in(0,1), 
\]
or, equivalently, if
\[
\frac{\displaystyle \int_0^p \frac{q}{g(G^{-1}(q))}dq}{\displaystyle \int_0^p \frac{q}{f(F^{-1}(q))}dq} \text{ is increasing in } p\in(0,1).
\]
\end{Def}

Among the previous stochastic orders, the following implications hold when the left extreme of the supports is equal to 0.  
\begin{equation}\label{cuadro}
\begin{tabular}{c c c c c }
  $X\leq_{\textup{c}}Y$ & $\Rightarrow$ & $X\leq_{\textup{qmit}}Y$ & $\Rightarrow$ & $X\leq_{\textup{*}}Y$  \\
 & \rotatebox{45}{$\Downarrow$} & &  & $\Downarrow$  \\
  &  & $X\leq_{\textup{dmrl}}Y$ & $\Rightarrow$ & $X\leq_{\textup{nbue}}Y.$  
\end{tabular}
\end{equation}

A natural question which arises looking at the definitions of the previous transform orders is how these criteria can be assured when the quantile functions and/or some incomplete integrals of the quantile functions do not have explicit expressions (or they are not easy to deal with). One possibility is to check the convex transform order because this criterion is stronger than the remaining transform orders (see \eqref{cuadro}). However, there are situations where the convex order does not hold. Let us see some examples.

Let us recall that $X\le_{\textup{c}}Y$ holds if, and only if the transformation $G^{-1}(F(x))$ is convex over the support of $X$. Therefore, given two samples $(X_1,\ldots,X_n)$ and $(Y_1,\ldots,Y_m)$ of $X$ and $Y$, respectively, we can have an approximation of the behaviour of $G^{-1}(F)$ by plotting the transformation of $G_m^{-1}(F_n)$, where $F_n$ and $G_m$ are the empirical distribution functions of $(X_1,\ldots,X_n)$ and $(Y_1,\ldots,Y_m)$, respectively.

Next, we consider the plots of $G_m^{-1}(F_n)$ for two real data sets. 

In the first example, we analyze a well-known data set, where $X$ and $Y$ denote the lifetimes of two populations of guinea pigs (see Bjerkedal, 1960). The first sample corresponds to a treatment group ($X$) which received a dose of tubercle bacilli, and the second one  to a control group ($Y$). As we can see in Figure \ref{example-unimodality} (on the left), the plot is apparently convex initially, and later is concave.

In the second example, we consider a data set of two populations of lifetimes of SPF Fisher 344 male rats (Yu et al., 1982). The first sample corresponds to a control group ($X$) under an \textit{ad libitum} (non restricted) diet, and the second one  to a treatment group ($Y$) with a restricted diet. Again, the plot is apparently convex initially, and later is concave (see Figure \ref{example-unimodality}, on the right).

\begin{figure}[ht]\centering
\includegraphics[width=\linewidth]{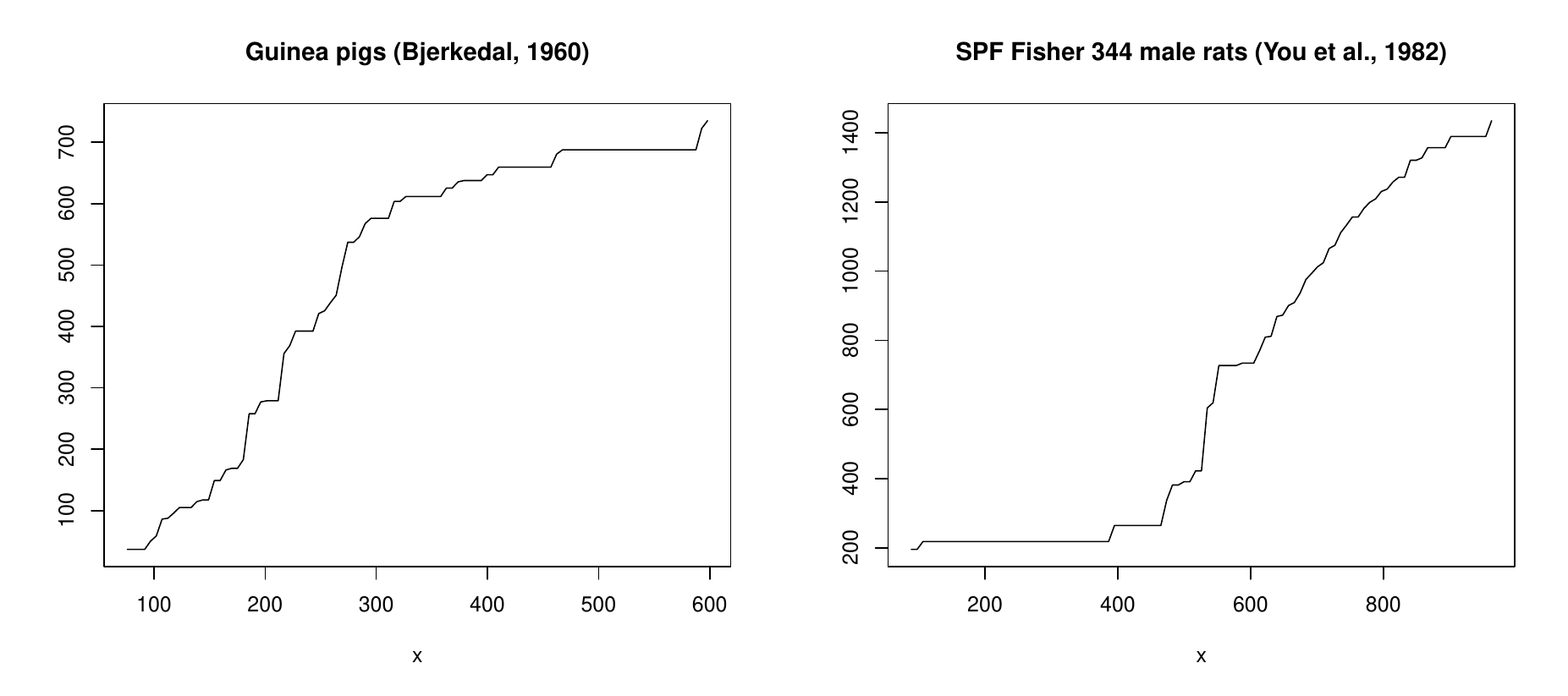}
\caption{\label{example-unimodality}\textit{Plots of the empirical transform $G_m^{-1}(F_n)$ for guinea pigs (on the left) and SPF Fisher 344 male rats (on the right).}}
\end{figure}

Observe that the property of  $G^{-1}(F)$ being initially convex and later concave is equivalent to the unimodality of the function $f(F^{-1}(p))/g(G^{-1}(p))$, therefore, the main purpose of this paper is to describe situations where the ratio of the quantile density functions has a relative extreme (in other words, the convex order does not hold), but it can still be ensured that the random variables are ordered in the star-shaped, the dmrl and/or the qmit order, as well as in the expected proportional shortfall order that will be introduced later. This work continues with some previous ones on sufficient conditions for other stochastic orders based on  quantile functions (see Arnold and Sarabia, 2018, Belzunce et. al, 2016 and 2017 and Belzunce and Mart\'{i}nez-Riquelme, 2017 for further details). Furthermore, the general situation, where the convex transform order does not hold, is also considered, that is, where the ratio of the quantile densities has at least one relative extreme.

The paper is organized as follows. In Section \ref{S2}, we present a first result where the number of modes of the quantile densities ratio is related to the number of modes of the ratio of the corresponding quantile functions. Later, this result is applied to provide sufficient conditions for the star-shaped and the expected proportional shortfall order. Similar results for the qmit and the dmrl orders are also established. These results are applied in Section \ref{S3} to compare two Tukey generalized distributed random variables in terms of some of the mentioned criteria. In addition, the main theorems are applied to establish new relationships among non-monotone and positive aging notions.

\section{Sufficient conditions for the transform orders}\label{S2}

As it was pointed out in the introduction, we deal with situations where the convex transform order does not hold. In that context, we try to find those cases where some other weaker transform order is still satisfied. Since we are interested in the situations where the convex transform order does not hold, we assume that the ratio of the quantile densities has a finite number of modes ($n\geq 1$). As a first step, we relate the number of modes of the ratio of the quantile densities to the number of modes of the ratio of the quantiles. A tool that will be applied in the proof of the next theorem is the number of sign changes of a function. We describe this next.

Given a function $f$, the sign changes of $f$ on an interval $I\subseteq \mathbb{R}$, denoted by $S^{-}(f)$, is defined as
\[
S^{-}(f)=\sup\{S^{-}(f(x_{1}),\dots,f(x_{m}))\},
\]
where $S^{-}(f(x_{1}),\dots,f(x_{m}))$ denotes the sign changes in this sequence and the supreme is taken among every ordered sequence in $I$, that means, for all $x_{1}<\cdots <x_{m}$, for all $m\in\mathbb{N}$.

Now, let us state the first result. 

\begin{The} \label{Th2.0}
Let $X$ and $Y$ be two non-negative random variables with differentiable quantile functions $F^{-1}$ and $G^{-1}$ and density functions $f$ and $g$, respectively. Let us assume that the ratio $f(F^{-1}(p))/g(G^{-1}(p))$ has n modes at $0<p_1 < p_2 < \ldots < p_n < 1$, such that the ratio is initially increasing. Let us define the function 
\begin{equation}\label{delta}
\delta (p) = F^{-1}(p)\frac{f(F^{-1}(p))}{g(G^{-1}(p))} - G^{-1}(p),
\end{equation}
for all $p\in(0,1)$. Then, the number of modes of $G^{-1}(p)/F^{-1}(p)$ is less or equal that $n+1$. In particular, the number of modes depends on the following conditions:
\begin{itemize}
  \item There is a mode at the interval  $(0,p_1)$ if, and only if, $ \lim_{p \rightarrow 0^+} \delta (p) \delta(p_1) <0$. 
  \item There is a mode at any interval $(p_i, p_{i+1})$, for $i=1,\ldots, n-1$ if, and only if, $ \delta(p_{i})\delta(p_{i+1}) <0$. 
  \item There is a mode at the interval  $(p_{n},1)$ if, and only if, $\lim_{p \rightarrow 1^-} \delta (p) \delta(p_n) <0$.
\end{itemize}
\end{The}

\begin{proof} In order to set the number of modes of the ratio $G^{-1}(p)/F^{-1}(p)$,  we seek for the sign changes of its derivative, which is equivalent to study
\[
S^{-}(\delta ), \text{ on the interval } (0,1).
\]
Now, observe that $\delta$ can be rewritten as
\[
\delta (p) = \int_{0}^{p} \left[\frac{1}{f(F^{-1}(q))} \frac{f(F^{-1}(p))}{g(G^{-1}(p))} -\frac{1}{g(G^{-1}(q))} \right]dq + l_X\frac{f(F^{-1}(p))}{g(G^{-1}(p))} -l_Y,
\]
where the equivalence follows taking into account that $F^{-1}(p)=\int_{0}^{p} \frac{1}{f(F^{-1}(q))}dq + l_X$ and analogously for $G^{-1}(p)$.  On the one hand, it is easy to see that 
\[l_X\frac{f(F^{-1}(p))}{g(G^{-1}(p))} -l_Y
\]
is increasing (decreasing) over the intervals where $f(F^{-1}(p))/g(G^{-1}(p))$ is increasing (decreasing), since $l_X \geq 0$. On the other hand, if we consider $ p< p' $ belonging to an interval where the ratio of the quantile densities is decreasing (increasing), we have that
\begin{multline*}
\int_0^{p} \left(\frac{f(F^{-1}(p))}{g(G^{-1}(p))}\frac{1}{f(F^{-1}(q))} - \frac{1}{g(G^{-1}(q))} \right)dq \\ \ge (\le) \int_0^{p} \left(\frac{f(F^{-1}(p'))}{g(G^{-1}(p'))}\frac{1}{f(F^{-1}(q))} - \frac{1}{g(G^{-1}(q))} \right)dq \\ \ge (\le) \int_0^{p'} \left(\frac{f(F^{-1}(p'))}{g(G^{-1}(p'))}\frac{1}{f(F^{-1}(q))} - \frac{1}{g(G^{-1}(q))} \right)dq, 
\end{multline*}
where both inequalities follow by the decrease (increase) of $f(F^{-1}(p))/g(G^{-1}(p))$. Therefore,  $\delta (p)$ is decreasing (increasing) over the intervals where $f(F^{-1}(p))/g(G^{-1}(p))$ is decreasing (increasing) and the infimums of $\delta$ are $0, p_2, p_4, \ldots, p_{2m}$, if $n=2m$, and $0, p_2, p_4, \ldots, p_{2m}, 1$ when $n=2m+1$. The result trivially follows taking into account that the conditions appearing in the statement are equivalent to ask $\delta $ for being negative at its infimums and positive at its supremums, which trivially implies a sign change or, equivalently, a mode of the ratio $G^{-1}(p)/F^{-1}(p)$.
\end{proof}

Whenever two random variables can not be ordered according to the convex criterion, then the ratio of the quantile densities has to have $n\geq 1$ modes. According to the previous theorem, if $n$ is finite, we know exactly the behavior of the ratio of the quantiles in such situations. In particular,  if at least one of the conditions stated in Theorem \ref{Th2.0} is satisfied, then the ratio of the quantiles has one or more modes. In this case, by applying some of the results provided by Belzunce et al. (2017), we can directly conclude that the only transform order which can hold is the expected proportional shortfall (see Definition \ref{eps}). Belzunce et al. provide the situations where the ratio of the quantiles has $n$ modes and the expected proportional shortfall order is satisfied.  In contrast with that situation, the star shaped order is trivially satisfied if $ \delta$ is non-negative at its infimums, so the following corollary follows as a direct consequence of the previous theorem.

\begin{Cor}
Let $X$ and $Y$ be two non-negative random variables with differentiable quantile functions $F^{-1}$ and $G^{-1}$ and density functions $f$ and $g$, respectively. Let us assume that the ratio $f(F^{-1}(p))/g(G^{-1}(p))$ has n modes at $0<p_1 < p_2 < \ldots < p_n < 1$, such that the ratio is initially increasing. Let us consider the function $\delta$ defined in \eqref{delta}. Then, $X\le_{\star} Y$ if, and only if, one of the following sets of conditions is satisfied
\begin{itemize}
\item $n=2m$ and $ \lim_{p \rightarrow 0^+} \delta (p) \geq  0$ and  $ \delta (p_{2\,i}) \geq  0$, for $i=1,\ldots, m$. 
\item $n=2m+1$ and $ \lim_{p \rightarrow 0^+} \delta (p) \geq  0$, $ \delta (p_{2\,i}) \geq 0$, for $i=1,\ldots, m$ and $ \lim_{p \rightarrow 1^-} \delta (p)\geq 0$.
\end{itemize}
Similarly, $X\ge_{\star}Y$ if, and only if, one the following sets of conditions is satisfied
\begin{itemize}
\item $n=2m$ and $ \lim_{p \rightarrow 1^-} \delta (p) \leq  0$ and  $ \delta (p_{2\,i-1}) \leq  0$, for $i=1,\ldots, m$. 
\item $n=2m+1$ and $ \delta (p_{2\,i-1}) \leq 0$, for $i=1,\ldots, m+1$.
\end{itemize}
\end{Cor}

Despite Theorem \ref{Th2.0} being interesting from a theoretical point of view, the conditions on the function $\delta$ are not usually easy to deal with and, consequently, this result is not so easy to apply in practice. Taking into account that downside, we consider the particular case where the ratio of the quantile densities has just one mode. This choice has a double justification. First, the conditions arising in that case are easy to check. Secondly, this situation occurs many times in practice as we have shown in the introduction. Our aim is to provide simpler sets of sufficient conditions for the expected proportional shortfall, the star-shaped, the qmit and the dmrl orders in such situations where the ratio of the quantile densities has just one mode. These sets of sufficient conditions are based on a former result in the same spirit of Theorem \ref{Th2.0}.

\begin{The} \label{Th2.1}
Let $X$ and $Y$ be two non-negative random variables with differentiable quantile functions $F^{-1}$ and $G^{-1}$ and density functions $f$ and $g$, respectively. Let us assume that there exists a point $p^* \in (0,1)$, such that $f(F^{-1}(p))/g(G^{-1}(p))$ is increasing over $\left(0,p^*\right)$ and decreasing over $\left(p^*,1\right)$. Let us consider the function $\delta$ defined in \eqref{delta}. Then, the behaviour of the ratio of the quantile functions is described by one, and just one, of the following cases:
\begin{enumerate}

\item If $\lim_{p \rightarrow 0^+} \delta(p) \geq 0$ and $\lim_{p \rightarrow 1^-} \delta(p) \geq 0$, then
$G^{-1}(p)/F^{-1}(p)$ is increasing over $\left(0,1\right)$.

\item If\, $\lim_{p \rightarrow 0^+} \delta(p) \geq 0$ and\, $\lim_{p \rightarrow 1^-} \delta(p) < 0$, then
there exist $p_1$ such that $G^{-1}(p)/F^{-1}(p)$ is increasing over $\left(0,p_1\right)$ and decreasing over $\left(p_1,1\right)$.

\item If $\lim_{p \rightarrow 0^+} \delta(p) < 0$ and $\lim_{p \rightarrow 1^-} \delta(p) \geq 0$, then
there exist $p_1$ such that $G^{-1}(p)/F^{-1}(p)$ is decreasing over $\left(0,p_1\right)$ and increasing over $\left(p_1,1\right)$.

\item If $\lim_{p \rightarrow 0^+} \delta(p) < 0$ and $\lim_{p \rightarrow 1^-} \delta(p) < 0$, then
	\begin{enumerate}
	
	\item If $\delta(p^*) \leq 0$, then $G^{-1}(p)/F^{-1}(p)$ is decreasing over $\left(0,1\right)$.
	\item If $\delta(p^*) > 0$, then there exist $p_1,p_2$ such that $G^{-1}(p)/F^{-1}(p)$ is decreasing over $\left(0,p_1\right)$, increasing over $\left(p_1,p_2\right)$ and decreasing over $\left(p_2,1\right)$.
	
	\end{enumerate}

\end{enumerate}
\end{The}

\begin{proof} 
According to the notation followed along the proof of Theorem \ref{Th2.0}, the following condition holds under the assumptions
\[
S^{-}(\delta) \leq 2, \text{  on the interval } (0,1),
\]
with sign sequence $(-,+,-)$, when the equality holds.  Let us discuss now the different cases considered in the statement of the theorem.

\begin{enumerate}

\item If $\lim_{p\rightarrow 0^{+}} \delta(p) \geq 0$ and $\lim_{p\rightarrow 1^{-}} \delta(p) \geq 0$ and taking into account that the functions $\delta(p)$ and $f(F^{-1}(p))/g(G^{-1}(p))$ share the same type of monotonicity, then we have that $S^{-}(\delta)=0, \text{  on the interval } (0,1),$ with the sign $(+)$ or, equivalently, the ratio $G^{-1}(p)/F^{-1}(p)$ is increasing over $(0,1)$ and the proof of Case 1 is concluded. 

\item If $\lim_{p\rightarrow 0^{+}} \delta(p) \geq 0$ and $\lim_{p\rightarrow 1^{-}} \delta(p) < 0$, then $S^{-}(\delta)=1, \text{   on the interval } (0,1),$ with the sign sequence $(+,-)$ or, equivalently, there exist $p_1$ such that the ratio $G^{-1}(p)/F^{-1}(p)$ is increasing over $(0,p_1)$ and decreasing over $(p_1,1)$ and the proof of Case 2 is concluded. 

\item If $\lim_{p\rightarrow 0^{+}} \delta(p) < 0$ and $\lim_{p\rightarrow 1^{-}} \delta(p) \geq 0$, then $S^{-}(\delta)=1, \text{   on the interval } (0,1),$ with the sign sequence $(-,+)$ or, equivalently, there exist $p_1$ such that the ratio $G^{-1}(p)/F^{-1}(p)$ is decreasing over $(0,p_1)$ and increasing over $(p_1,1)$ and the proof of Case 3 is concluded. 

\item Finally, if $\lim_{p\rightarrow 0^{+}} \delta(p) < 0$ and $\lim_{p\rightarrow 1^{-}} \delta(p) < 0$, we have to distinguish two cases. 

\begin{enumerate}

\item If $\delta(p^*) \leq 0$, then $S^{-}(\delta)=0, \text{   on the interval } (0,1),$ with the sign $(-)$ or, equivalently, the ratio $G^{-1}(p)/F^{-1}(p)$ is decreasing over $(0,1)$.

\item Otherwise, if  $\delta(p^*) > 0$, then $S^{-}(\delta)=2, \text{   on the interval } (0,1),$ with the sign sequence $(-,+,-)$ or, equivalently, there exist $p_1,p_2\in(0,1)$ such that the ratio $G^{-1}(p)/F^{-1}(p)$ is decreasing over $\left(0,p_1\right)$, increasing over $\left(p_1,p_2\right)$ and decreasing over $\left(p_2,1\right)$ and the proof of Case 4 is concluded. 

\end{enumerate}

\end{enumerate}
\end{proof}

The previous theorem provides a full discussion about the behaviour of the ratio of the quantile functions when there exists a point $p^* \in (0,1)$, such that $f(F^{-1}(p))/g(G^{-1}(p))$ is increasing over $\left(0,p^*\right)$ and decreasing over $\left(p^*,1\right)$. 

We observe that an important point in order to apply the previous theorem is the limit of $\delta(p)$ as $p \rightarrow 1^-$ and $p \rightarrow 0^+$. These limits have to be computed for every particular case. Next, we provide some general remarks about these limits.

\begin{Rem}According to Parzen (1979) and Shuster (9184), the quantile density function of a random variable $X$ has the following representation
\[
f(F^{-1}(p))=L_F(p) (1-p)^{\alpha_F},
\] 
as $p \rightarrow 1^-$, where $L_F$ is a slowly varying function as $p\rightarrow 1^-$ and $\alpha_F$ is a positive constant. We recall that $L_F(p)$ is slowly varying as $p\rightarrow 1^-$ if $\lim_{p\rightarrow 1^-} L_F(\lambda p)/L_F(p)=1$ for every $\lambda >0$. 

Taking this into account, we have that 
\[
\lim_{p\rightarrow 1^-}\frac{f(F^{-1}(p))}{g(G^{-1}(p))}= \lim_{p\rightarrow 1^-}\frac{L_F(p)}{L_G(p)}(1-p)^{\alpha_F-\alpha_G},
\]
therefore, the previous limit exists and would be either 0, or a positive constant $c$, or $+\infty$. 
 
From the previous remark, and depending on the upper left extremes of the supports of $X$ and $Y$, we can make the following considerations on the limit of $\delta(p)$ as $p\rightarrow 1^-$. 

\textbf{Case 1}: $\lim_{p\rightarrow 1^-} G^{-1}(p)=u_Y < +\infty$ and $\lim_{p\rightarrow 1^-} F^{-1}(p)= u_X < +\infty$. In this case we have that 
\[
\lim_{p\rightarrow 1^-} \delta(p)= \left\{
\begin{array}{ll}
u_Xc - u_Y & \text{if }\lim_{p\rightarrow 1^-}\frac{f(F^{-1}(p))}{g(G^{-1}(p))}=c \\
           &                       \\     
-u_Y & \text{if }\lim_{p\rightarrow 1^-}\frac{f(F^{-1}(p))}{g(G^{-1}(p))}=0 \\
           &                       \\     
+\infty & \text{if }\lim_{p\rightarrow 1^-}\frac{f(F^{-1}(p))}{g(G^{-1}(p))}=+\infty \\
\end{array}
\right.
\]

\textbf{Case 2}: $\lim_{p\rightarrow 1^-} G^{-1}(p)=u_Y < +\infty$ and $\lim_{p\rightarrow 1^-} F^{-1}(p)= +\infty$. In this case we have that 
\[
\lim_{p\rightarrow 1^-} \delta(p)= \left\{
\begin{array}{ll}
+\infty & \text{if }\lim_{p\rightarrow 1^-}\frac{f(F^{-1}(p))}{g(G^{-1}(p))}=c \\
           &                       \\     
\text{Indeterminate} & \text{if }\lim_{p\rightarrow 1^-}\frac{f(F^{-1}(p))}{g(G^{-1}(p))}=0 \\
           &                       \\     
+\infty & \text{if }\lim_{p\rightarrow 1^-}\frac{f(F^{-1}(p))}{g(G^{-1}(p))}=+\infty \\
\end{array}
\right.
\]

\textbf{Case 3}: $\lim_{p\rightarrow 1^-} G^{-1}(p)=+\infty$ and $\lim_{p\rightarrow 1^-} F^{-1}(p)=u_X< +\infty$. In this case we have that 
\[
\lim_{p\rightarrow 1^-} \delta(p)= \left\{
\begin{array}{ll}
-\infty & \text{if }\lim_{p\rightarrow 1^-}\frac{f(F^{-1}(p))}{g(G^{-1}(p))}=c \\
           &                       \\     
-\infty & \text{if }\lim_{p\rightarrow 1^-}\frac{f(F^{-1}(p))}{g(G^{-1}(p))}=0 \\
           &                       \\     
\text{Indeterminate} & \text{if }\lim_{p\rightarrow 1^-}\frac{f(F^{-1}(p))}{g(G^{-1}(p))}=+\infty \\
\end{array}
\right.
\]
 
\textbf{Case 4}: $\lim_{p\rightarrow 1^-} G^{-1}(p)=+\infty$ and $\lim_{p\rightarrow 1^-} F^{-1}(p)=+\infty$. In any case, the limit is indeterminate.
 
Similar comments can be made when $p \rightarrow 0^+$, see Alzaid and Al-Osh (1989) for more details. In particular, we have the following cases: 

\textbf{Case 1'}: $\lim_{p\rightarrow 0^+} F^{-1}(p)= l_X > 0$. In this case we have that 
\[
\lim_{p\rightarrow 0^+} \delta(p)= \left\{
\begin{array}{ll}
l_Xc' - l_Y & \text{if }\lim_{p\rightarrow 0^+}\frac{f(F^{-1}(p))}{g(G^{-1}(p))}=c' \\
           &                       \\     
-l_Y & \text{if }\lim_{p\rightarrow 0^+}\frac{f(F^{-1}(p))}{g(G^{-1}(p))}=0 \\
           &                       \\     
+\infty & \text{if }\lim_{p\rightarrow 0^+}\frac{f(F^{-1}(p))}{g(G^{-1}(p))}=+\infty \\
\end{array}
\right.
\]

\textbf{Case 2'}: $\lim_{p\rightarrow 0^+} F^{-1}(p)=  0$. In this case we have that 
\[
\lim_{p\rightarrow 0^+} \delta(p)= \left\{
\begin{array}{ll}
 - l_Y & \text{if }\lim_{p\rightarrow 0^+}\frac{f(F^{-1}(p))}{g(G^{-1}(p))}=c' \\
            &                       \\     
-l_Y & \text{if }\lim_{p\rightarrow 0^+}\frac{f(F^{-1}(p))}{g(G^{-1}(p))}=0 \\
           &                       \\     
\text{Indeterminate} & \text{if }\lim_{p\rightarrow 0^+}\frac{f(F^{-1}(p))}{g(G^{-1}(p))}=+\infty \\
\end{array}
\right.
\]
\end{Rem}

Now, two sets of sufficient conditions for the star-shaped order are provided by applying Case 1 and Case 4a in Theorem \ref{Th2.1}.

\begin{The}\label{Th2.2} Let $X$ and $Y$ be two non-negative random variables, with quantile and density functions $F^{-1}$, $f$ and $G^{-1}$, $g$, respectively, where the quantile functions are differentiable. Let us assume that there exists a point $p^* \in (0,1)$, such that $f(F^{-1}(p))/g(G^{-1}(p))$ is increasing over $\left(0,p^*\right)$ and decreasing over $\left(p^*,1\right)$. 

\begin{enumerate}

\item $X\le_{\star}Y$ if, and only if, $\lim_{p \rightarrow 0^+} \delta(p) \geq 0$ and $\lim_{p \rightarrow 1^-} \delta(p) \geq 0$. 

\item $X\ge_{\star}Y$ if, and only if, $\delta(p^*) \leq 0$.

\end{enumerate}
\end{The}

Additionally, it is possible to provide results for another stochastic order based on the quantile function which was introduced by Belzunce et. al (2012). 
\begin{Def}\label{eps}
Given two random variables $X$ and $Y$ with distribution functions $F$ and $G$, respectively, it is said that $X$ is smaller than $Y$ in the \textbf{expected proportional shortfall order}, denoted by $X\le_{ps}Y$, if 
\[
\frac{\int_{F^{-1}(p)}^{+\infty}\overline F(x)dx}{F^{-1}(p)}\le \frac{\int_{G^{-1}(p)}^{+\infty}\overline G(x)dx}{G^{-1}(p)},\text{ for all }p\in(0,1).
\]
\end{Def}
Recall that the function
\[
EPS_X(p)=\frac{\int_{F^{-1}(p)}^{+\infty}\overline F(x)dx}{F^{-1}(p)},\text{ for }p\in(0,1),
\]
is known in risk theory as the expected proportional shortfall. The expected proportional shortfall  can be considered as a risk measure which does not depend on the base currency, and can also be considered as a measure of relative deprivation. Therefore, the eps order can be used, on the one hand, to compare risks and, on the other hand, to compare income distributions according to their concentration or inequality. According to \eqref{cuadro} and well-known results for the ps order (see Belzunce et. al, 2012), we have
\[
\begin{tabular}{c c c c c  }
    &    & $X\leq_{\textup{qmit}}Y$ &  $\Rightarrow$ &  $X\leq_{\textup{*}}Y$  \\ 
    & \rotatebox{45}{$\Rightarrow$} &   &  &  $\Downarrow$ \\ 
$X\leq_{\textup{c}}Y$ &   &   &   & $X\leq_{\textup{ps}}Y$  \\
    & \rotatebox{45}{$\Downarrow$} &     &  &  $\Downarrow$ \\ 
    &    & $X\leq_{\textup{dmrl}}Y$ &  $\Rightarrow$ &    $X\leq_{\textup{nbue}}Y$.  \\ 
\end{tabular}
\]

Recently, Belzunce and Mart\'{i}nez-Riquelme (2017) describe sufficient conditions where the ps order holds when the ratio of the quantile functions is non-monotone. Since Theorem \ref{Th2.1} deals with the non monotonicity of the ratio of  quantile functions, it can be jointly applied with some of the results provided by Belzunce and Mart\'{i}nez-Riquelme (2017) to establish sufficient conditions for the ps order when the quantile densities ratio is unimodal. Let us state the sufficient conditions for the ps order in the same terms as in Theorem \ref{Th2.1}.

\begin{The}\label{Thps} Let $X$ and $Y$ be two non-negative random variables, with quantile and density functions $F^{-1}$, $f$ and $G^{-1}$, $g$, respectively, where the quantile functions are differentiable. Let us assume that there exists a point $p^* \in (0,1)$, such that $f(F^{-1}(p))/g(G^{-1}(p))$ is increasing over $\left(0,p^*\right)$ and decreasing over $\left(p^*,1\right)$ and define the function
\begin{equation}\label{deltaps}
\delta_{ps} (p) =\frac{F^{-1}(p)}{E[X]} - \frac{G^{-1}(p)}{E[Y]}.
\end{equation}

\begin{enumerate}
\item $X\leq_{\textup{ps}}Y$ if, and only if, $\lim_{p \rightarrow 0^{+}} \delta(p) < 0$, $\lim_{p \rightarrow 1^{-}} \delta(p) \geq 0$ and $\lim_{p \rightarrow 0^{+}} \delta_{ps}(p) \geq 0$.

\item $X\geq_{\textup{ps}}Y$ if, and only if, one of the following conditions hold
\begin{enumerate}
	\item $\lim_{p \rightarrow 0^{+}} \delta(p) \geq 0$, $\lim_{p \rightarrow 1^{-}} \delta(p) < 0$ and $\lim_{p \rightarrow 0^{+}} \delta_{ps}(p) \leq 0$.
  \item $\lim_{p \rightarrow 0^{+}} \delta(p) < 0$, $\lim_{p \rightarrow 1^{-}} \delta(p) < 0$, $\delta(p^*) > 0$, $\lim_{p \rightarrow 0^{+}} \delta_{ps}(p) \leq 0$ and \newline $EPS_{X}(p_1) \geq EPS_{Y}(p_1)$, where $p_1$ denotes the minimum of the ratio of the quantiles functions.
\end{enumerate}
\end{enumerate}

\end{The}

\begin{proof}

The proof of Case 1 and Case 2.a follow from Theorem \ref{Th2.1} and by applying Theorem 3.12 in Belzunce et. al (2017). Let us consider now Case 2.b. From Theorem \ref{Th2.1}, we have that there exist $p_1,p_2 \in (0,1)$ such that $G^{-1}(p)/F^{-1}(p)$ is decreasing over $(0,p_1)$, increasing over $(p_1,p_2)$ and decreasing over $(p_2,1)$. According to the proof of Theorem 4.5 in Belzunce and Mart\'{i}nez-Riquelme (2017), under these conditions, the function $G^{-1}(p)(EPS_Y (p)-EPS_X (p))$ is increasing over $(0,p_1)$, decreasing over $(p_1,p_2)$ and increasing over $(p_2,1)$. Moreover, $G^{-1}(p)(EPS_Y (p)-EPS_X (p))$ ends non-positive. Therefore, $X\geq_{\textup{ps}}Y$ holds if, and only if, $G^{-1}(p)(EPS_Y (p)-EPS_X (p))$ is non-positive at its maximum $p_1$ or, equivalently, if $EPS_{X}(p_1) \geq EPS_{Y}(p_1)$. 

\end{proof}

Since the qmit order is weaker than the convex order but stronger than the star shaped one, the following question arises in the context of Theorem \ref{Th2.0}. If  the ratio of the quantile densities has $n\geq 1$ modes, is there any situation where the qmit order still holds? The following result characterizes the cases to get a positive answer to this question when $n$ is finite.

\begin{The}\label{Thqmitn} 
Let $X$ and $Y$ be two random variables with interval supports and quantile and density functions $F^{-1}$, $f$ and $G^{-1}$, $g$, respectively. Assume that $\frac{f(F^{-1}(p))}{g(G^{-1}(p))}$ has a finite number of modes ($n\geq 1$) at $(p_1, \ldots, p_n)$, such that the first one occurs from increasing to decreasing. Then, $X\leq_{\textup{qmit}}Y$ if, and only if, one of the following sets of conditions is satisfied
\begin{itemize}
\item  $n=2m$ and
\begin{equation}\label{even_cond}
\int_{0}^{p_{2i}} q\left(\frac{f(F^{-1}(p_{2i}))}{g(G^{-1}(p_{2i}))} \frac{1}{f(F^{-1}(q))} - \frac{1}{g(G^{-1}(q))} \right)dq \geq 0, \text{ for } i =1, \ldots, m.
\end{equation}
\item  $n=2m+1$ and 
\begin{equation}\label{odd_cond1}
\int_{0}^{p_{2i}} q\left(\frac{f(F^{-1}(p_{2i}))}{g(G^{-1}(p_{2i}))}\frac{1}{f(F^{-1}(q))} - \frac{1}{g(G^{-1}(q))} \right)dq \geq 0,  \text{ for }  i =1, \ldots, m  
\end{equation}
\noindent and
\begin{equation}\label{odd_cond12}
\lim_{p\rightarrow 1^-} \frac{f(F^{-1}(p))}{g(G^{-1}(p))}(F^{-1}(p)-E[X])-(G^{-1}(p)-E[Y]) \ge 0. 
\end{equation}
\end{itemize}
\end{The}

\begin{proof} According to the definition of the qmit order, $X\leq_{\textup{qmit}}Y$ holds if, and only if,
\[
\frac{\displaystyle \int_0^p \frac{q}{g(G^{-1}(q))}dq}{\displaystyle \int_0^p \frac{q}{f(F^{-1}(q))}dq} \text{ is increasing in } p\in(0,1).
\]
Under the assumptions, the previous ratio is differentiable and, by taking derivatives, the condition above is equivalent to
\[
\delta_{qmit}(p)\equiv \int_0^p q\left(\frac{f(F^{-1}(p))}{g(G^{-1}(p))}\frac{1}{f(F^{-1}(q))} - \frac{1}{g(G^{-1}(q))} \right)dq\ge 0,\text{ for all } p\in(0,1).
\]
Following analogous steps to the proof of Theorem \ref{Th2.0}, it is easy to see that $\delta_{qmit}(p)$ is increasing (decreasing) over the intervals where $\frac{f(F^{-1}(p))}{g(G^{-1}(p))}$ is increasing (decreasing). The proof follows observing that Condition \eqref{even_cond} is equivalent to ask for $\delta_{qmit}$ being non-negative at its minimums, as well as \eqref{odd_cond1} and \eqref{odd_cond12}. The especial cases are the first and the last intervals, which we detail next. On the one hand, $\delta_{qmit}(p)\ge 0$, for all $ p\in(0, p_1)$, since the integrand is always non-negative over this interval by the assumptions. On the other hand, $\delta_{qmit}$ ends decreasing in the odd case and, therefore, the condition $\lim_{p\rightarrow 1^-} \delta_{qmit} (p)\geq 0 $ is required or, equivalently,
$$ \lim_{p\rightarrow 1^-} \frac{f(F^{-1}(p))}{g(G^{-1}(p))}(F^{-1}(p)-E[X])-(G^{-1}(p)-E[Y]) \ge 0, $$
that is, if Condition \eqref{odd_cond12} holds.
\end{proof}

Once again, the general conditions appearing in the statement of the previous theorem are not easy to deal with, so the interest of that result is rather theoretical. The main problem is that the modes must be known to check Condition \eqref{even_cond} or Condition \eqref{odd_cond1} and \eqref{odd_cond12} depending on the case. In contrast with that general case, the particular case $n=1$ does not require knowing the mode, because the unique condition which has to be verified is \eqref{odd_cond12}. Therefore, the following corollary about the particular case $n=1$ has interest from a practical point of view, because that case often appears in practice, as we have already pointed out.

\begin{Cor}\label{Th2.3} Let $X$ and $Y$ be two non-negative random variables with quantile and density functions $F^{-1}$, $f$ and $G^{-1}$, $g$, respectively, where the quantile functions are differentiable.  Let us assume that there exists a point $p^* \in (0,1)$, such that $f(F^{-1}(p))/g(G^{-1}(p))$ is increasing over $\left(0,p^*\right)$ and decreasing over $\left(p^*,1\right)$. Then, $X\leq_{\textup{qmit}}Y$ if, and only if, 
\begin{equation}\label{lim_cond}
\lim_{p\rightarrow 1^-} \frac{f(F^{-1}(p))}{g(G^{-1}(p))}(F^{-1}(p)-E[X])-(G^{-1}(p)-E[Y]) \ge 0,
\end{equation}
\end{Cor}

\begin{Rem} Recently, Kayid et. al (2018), provided a set of conditions for the qmit order. In particular, they assume that the ratio of the quantile densities is unimodal, as in the previous theorem, but they also assume the monotonicity of the function $(G^{-1}(p)-E[Y])/(F^{-1}(p)-E[X])$. It is obvious that Condition \eqref{lim_cond} is easier to check than the previous one. 
\end{Rem}

Next, we provide an analogous version of Theorem \ref{Thqmitn} and Corollary \ref{Th2.3} for the dmrl order. 

\begin{The}\label{Thdmrln} 
Let $X$ and $Y$ be two random variables with interval supports and quantile and density functions $F^{-1}$, $f$ and $G^{-1}$, $g$, respectively. Assume that $\frac{f(F^{-1}(p))}{g(G^{-1}(p))}$ has a finite number of modes ($n\geq 1$) at $(p_1, \ldots, p_n)$, such that the last one occurs from decreasing to increasing. Then, $X\leq_{\textup{dmrl}}Y$ if, and only if, one of the following sets of conditions is satisfied
\begin{itemize}
\item  $n=2m$ and
\begin{equation}\label{even_cond2}
\int_{p_{2i-1}}^1 (1-q)\left(\frac{1}{g(G^{-1}(q))} - \frac{f(F^{-1}(p_{2i-1}))}{g(G^{-1}(p_{2i-1}))}  \frac{1}{f(F^{-1}(q))}\right)dq \geq 0, \text{ for } i =1, \ldots, m.
\end{equation}
\item  $n=2m+1$ and 
\begin{equation}\label{odd_cond2}
\int_{p_{2i}}^1 (1-q)\left(\frac{1}{g(G^{-1}(q))} - \frac{f(F^{-1}(p_{2i}))}{g(G^{-1}(p_{2i}))}  \frac{1}{f(F^{-1}(q))}\right)dq \geq 0, \text{ for }  i =1, \ldots, m 
\end{equation}
\noindent and
\begin{equation}\label{odd_cond22}
\lim_{p\rightarrow 0^+} \frac{f(F^{-1}(p))}{g(G^{-1}(p))}(F^{-1}(p)-E[X])-(G^{-1}(p)-E[Y]) \ge 0. 
\end{equation}
\end{itemize}
\end{The}

\begin{proof} $X\leq_{\textup{dmrl}}Y$ holds if, and only if, 
\[
\frac{\displaystyle \int_p^1 \frac{1-q}{g(G^{-1}(q))}dq}{\displaystyle \int_p^1 \frac{1-q}{f(F^{-1}(q))}dq} \text{ is increasing in } p\in(0,1),
\]
\noindent or, equivalently, 
\[
\delta_{dmrl}(p)\equiv \int_p^1 (1-q)\left(\frac{1}{g(G^{-1}(q))} - \frac{f(F^{-1}(p))}{g(G^{-1}(p))}  \frac{1}{f(F^{-1}(q))}\right)dq\ge 0,\text{ for all } p\in(0,1).
\]
Once again, it is easy to see that $\delta_{dmrl}(p)$ is increasing (decreasing) over the intervals where $\frac{f(F^{-1}(p))}{g(G^{-1}(p))}$ is decreasing (increasing). The proof follows observing that Condition \eqref{even_cond2} is equivalent to ask for $\delta_{dmrl}$ being non-negative at its minimums, as well as \eqref{odd_cond2} and \eqref{odd_cond22}. The especial cases are the first and the last intervals, which we detail next. On the one hand, $\delta_{dmrl}(p)\ge 0$, for all $ p\in(p_n, 1)$, since the integrand is always non-negative over this interval by the assumptions. On the other hand, $\delta_{dmrl}$ starts increasing in the odd case and, therefore, the condition $\lim_{p\rightarrow 0^+} \delta_{dmrl} (p)\geq 0 $ is required or, equivalently,
$$ \lim_{p\rightarrow 0^+} \frac{f(F^{-1}(p))}{g(G^{-1}(p))}(F^{-1}(p)-E[X])-(G^{-1}(p)-E[Y]) \ge 0, $$
that is, if Condition \eqref{odd_cond22} holds.
\end{proof}

Again, the particular case $n=1$ has a special interest. Therefore, we state the following corollary for that case.

\begin{Cor}\label{Th2.5} Let $X$ and $Y$ be two non-negative random variables with interval supports and quantile and density functions $F^{-1}$, $f$ and $G^{-1}$, $g$, respectively, where the quantile functions are differentiable.  Let us assume that there exists a value $p^* \in (0,1)$, such that $f(F^{-1}(p))/g(G^{-1}(p))$ is decreasing over $\left(0,p^*\right)$ and increasing over $\left(p^*,1\right)$. Then, $X\leq_{\textup{dmrl}}Y$ if, and only if, 
\begin{equation}\label{lim_conddmrl}
\lim_{p\rightarrow 0^+}  \frac{f(F^{-1}(p))}{g(G^{-1}(p))}(F^{-1}(p)-E[X])-(G^{-1}(p)-E[Y])\ge 0.
\end{equation}
\end{Cor}

\section{Applications}\label{S3}

In this section, we first summarize the results given along the previous section providing a whole discussion to compare two random variables in some of the transform orders (the star-shaped, the qmit, the dmrl and the ps order) in such situations where the ratio of the quantile density functions starts increasing and ends decreasing. Later, two Tukey generalized distributed random variables are ordered in some of the mentioned criteria according to the provided discussion. Finally, the results are applied to establish new relationships among non-monotone and positive aging notions.

\subsection{A summarized result to compare random variables in the transform orders}

Let us consider the different results that we have established in the previous section to simplify the comparison of two random variables in some of the transform orders. 

Let $X$ and $Y$ be two non-negative random variables with interval supports $Supp(X)=\left(l_X,u_X\right)$ and $Supp(Y)=\left(l_Y,u_Y\right)$, respectively, and quantile and density functions $F^{-1}$, $f$ and $G^{-1}$, $g$, respectively, where the quantile functions are differentiable. Let us assume that there exists a point $p^* \in (0,1)$, such that $f(F^{-1}(p))/g(G^{-1}(p))$ is increasing over $\left(0,p^*\right)$ and decreasing over $\left(p^*,1\right)$.  Then, the following algorithm sets the sufficient and necessary conditions for the qmit, the dmrl, the star-shaped and the ps orders.
\begin{enumerate}
	\item $X\leq_{\textup{qmit}}Y$ if, and only if, Condition \eqref{odd_cond12} holds.
	\item $X\geq_{\textup{dmrl}}Y$ if, and only if, Condition \eqref{odd_cond22} holds.
	\item $X\leq_{\textup{*}}Y$ if, and only if, $\lim_{p \rightarrow 0^{+}} \delta(p) \geq 0$ and $\lim_{p \rightarrow 1^-} \delta(p) \geq 0$.
	\item $X\geq_{\textup{*}}Y$ if, and only if, $ \delta(p^*) \leq 0$.
	\item $X\leq_{\textup{ps}}Y$ if, and only if, $\lim_{p \rightarrow 0^{+}} \delta(p) < 0$, $\lim_{p \rightarrow 1^{-}} \delta(p) \geq 0$ and $\lim_{p \rightarrow 0^{+}} \delta_{ps}(p) \geq 0$.
	\item $X\geq_{\textup{ps}}Y$ if, and only if, one of the following conditions hold
	\begin{enumerate}
	   \item $\lim_{p \rightarrow 0^{+}} \delta(p) \geq 0$, $\lim_{p \rightarrow 1^{-}} \delta(p) < 0$ and $\lim_{p \rightarrow 0^{+}} \delta_{ps}(p) \leq 0$.
	   \item $\lim_{p \rightarrow 0^{+}} \delta(p) < 0$, $\lim_{p \rightarrow 1^{-}} \delta(p) < 0$, $\delta(p^*) > 0$, $\lim_{p \rightarrow 0^{+}} \delta_{ps}(p) \leq 0$ and \newline $EPS_{X}(p_1) \geq EPS_{Y}(p_1)$, where $\delta $ is defined in \eqref{delta}, where $\delta_{ps}$ is defined in \eqref{deltaps}, $p_1$ denotes the minimum of the ratio of the quantiles functions and $EPS_{X}$ and $EPS_Y$ denote the expected proportional shortfall function of $X$ and $Y$, respectively.		
	\end{enumerate}
\end{enumerate}

Observe that the weaker orders also trivially hold when the stronger one is satisfied. Now, let us see some important remarks.

\begin{Rem} 
All the cases follow directly from the results given in the previous section except Case 2. The proof of this case is analogous to the one of Corollary \ref{Th2.5} taking into account that $f(F^{-1}(p))/g(G^{-1}(p))$ is decreasing over $\left(0,p^*\right)$ and increasing over $\left(p^*,1\right)$ and that $X\geq_{\textup{dmrl}}Y$ if, and only if, 
\[
\frac{\displaystyle \int_p^1 \frac{1-q}{g(G^{-1}(q))}dq}{\displaystyle \int_p^1 \frac{1-q}{f(F^{-1}(q))}dq} \text{ is decreasing in } p\in(0,1).
\]
\end{Rem}

\begin{Rem}
In such cases where there exists a value $p^* \in (0,1)$, such that the ratio $f(F^{-1}(p))/g(G^{-1}(p))$ is increasing over $\left(0,p^*\right)$ and decreasing over $\left(p^*,1\right)$ and $l_X \geq 0$, if the star shaped order holds and
\begin{equation}\label{1}
\lim_{p\rightarrow 1^{-}} \frac{f(F^{-1}(p))}{g(G^{-1}(p))} \leq \frac{E[Y]}{E[X]},
\end{equation}
\noindent then $X\leq_{\textup{qmit}}Y$. 
Therefore, in such a case, we obtain that both orders are equivalent. 
$$ X\leq_{\textup{*}}Y   \hspace{1cm} \Longleftrightarrow   \hspace{1cm}  X\leq_{\textup{qmit}}Y.$$
\end{Rem}

Whenever quantile density functions are unimodal, the above algorithm is especially useful for comparing random variables, provided they have a parametric model for their quantile density functions. Firstly, the situations where the ratio of the quantile density functions has just one maximum have to be identified. Then, the conditions \eqref{odd_cond12} and \eqref{odd_cond22} have to be checked to set the cases where the qmit and/or the dmrl order hold. Additionally, the limits of the $\delta$ function have to be checked to set the cases where the star-shaped order holds, as well as the sign of $\delta (p^*)$, if possible. Finally, the limits of the $\delta_{ps}$ function and the remaining conditions have to be checked to set the cases where the ps order holds.

As mentioned previously, let us now compare two Tukey generalized distributed random variables considering the qmit, the dmrl and the star-shaped orders.

\begin{Exam} Let us consider two Tukey generalized distributed random variables, that is, $X \sim T(\lambda_1, \eta_1, \alpha_1)$ and $Y\sim T(\lambda_2, \eta_2, \alpha_2)$ with quantile functions
\[
F^{-1}(p)= \lambda_1 + \eta_1 (p^{\alpha_1}-(1-p)^{\alpha_1}) \text{ , for all } p\in (0,1)
\]
and 
\[
G^{-1}(p)= \lambda_2 + \eta_2 (p^{\alpha_2}-(1-p)^{\alpha_2}) \text{ , for all } p\in (0,1),
\]
respectively, where $\lambda_1, \lambda_2, \eta_1, \eta_2, \alpha_1, \alpha_2 \in \mathbb{R}$ and $\eta_1, \eta_2 \neq0$. Since our results are given for non-negative random variables, we will assume that $l_X=\lambda_1- \eta_1 \ge 0$ and $l_Y=\lambda_2- \eta_2 \ge 0$.

According to the previous discussion, we first seek for the situations where the ratio $f(F^{-1}(p))/g(G^{-1}(p))$ has just one maximum. Along this discussion, we assume that $\alpha_2 \neq \alpha_1$, otherwise the discussion is trivial. In addition, we study separate the particular cases where $\alpha_1 \in \{1,2\}$ or $\alpha_2 \in \{1,2\}$, since they are different from the rest. Let us see first the case $\alpha_1 \in \{1,2\}$, where
\[
\frac{f(F^{-1}(p))}{g(G^{-1}(p))}= \frac{\eta_2 \alpha_2 (p^{\alpha_2-1}+ (1-p)^{\alpha_2-1})}{2 \eta_1 }.
\]
It is easy to see that the previous function has just one maximum on the interval $(0,1)$ if, and only if, $\alpha_2 \in (1,2)$. Under the assumptions, we have
\begin{itemize}
	\item $ \lim_{p \rightarrow 1^-} \frac{f(F^{-1}(p))}{g(G^{-1}(p))}(F^{-1}(p)-E[X])-(G^{-1}(p)-E[Y]) = \eta_2(\frac{\alpha_2}{2}-1) < 0$.
	\item $\lim_{p \rightarrow 0^+} \frac{f(F^{-1}(p))}{g(G^{-1}(p))}(F^{-1}(p)-E[X])-(G^{-1}(p)-E[Y]) = \eta_2(1-\frac{\alpha_2}{2}) > 0$.
	\item $ \lim_{p \rightarrow 1^-} \delta(p) = \frac{(\lambda_1+\eta_1) \eta_2 \alpha_2}{ 2 \eta_1 }- (\lambda_2+\eta_2)$.
	\item $\lim_{p \rightarrow 0^+} \delta(p) =  \frac{(\lambda_1-\eta_1) \eta_2 \alpha_2}{ 2 \eta_1 }- (\lambda_2-\eta_2)$.
\end{itemize}
Moreover, it is easy to see that the condition $ \lim_{p \rightarrow 1^-} \delta(p) \geq 0$ implies $ \lim_{p \rightarrow 0^+} \delta(p) \geq 0$, since
\begin{equation}\label{11}
0\leq \lim_{p \rightarrow 1^-} \delta(p) = \lim_{p \rightarrow 0^+} \delta(p) + \eta_2 (\alpha_2-2) \leq \lim_{p \rightarrow 0^+} \delta(p).
\end{equation}
To sum up, we can state the following. Let $X \sim T(\lambda_1, \eta_1, \alpha_1)$ and $Y\sim T(\lambda_2, \eta_2, \alpha_2)$ such that $\alpha_1 \in \{1,2\}$ and $\alpha_2 \in (1,2)$, then
\begin{itemize}
	\item The qmit order does not hold.
	\item The dmrl order does not hold.
  \item $X\leq_{\textup{*}}Y$ if, and only if, $ (\lambda_1+\eta_1)\eta_2 \alpha_2 > (\lambda_2+\eta_2) 2 \eta_1$.
  \item $X\geq_{\textup{*}}Y$ if, and only if, $ (\lambda_1-\eta_1)\eta_2 \alpha_2 < (\lambda_2-\eta_2) 2 \eta_1$.
\end{itemize}

Analogously, we can also state the following. Let $X \sim T(\lambda_1, \eta_1, \alpha_1)$ and $Y\sim T(\lambda_2, \eta_2, \alpha_2)$ such that $\alpha_2 \in \{1,2\}$, then
\begin{itemize}
	\item The qmit order does not hold.
	\item The dmrl order does not hold.
  \item $X\geq_{\textup{*}}Y$ if, and only if, $\alpha_1 \in (0,1)$.
  \item $X\leq_{\textup{*}}Y$ if, and only if, $\alpha_1 \in (2,+\infty)$ and $(\lambda_1+\eta_1)2 \eta_2 \geq (\lambda_2+\eta_2) \alpha_1 \eta_1$.
\end{itemize}

Now, let assume that $\alpha_1, \text{  }\alpha_2 \notin \{1,2\}$ and $\alpha_1 \neq \alpha_2$. The behavior of $f(F^{-1}(p))/g(G^{-1}(p))$ is equivalent to the one of  
\[
h(p) = \frac{p^{\alpha_2-1}+ (1-p)^{\alpha_2-1}}{p^{\alpha_1-1}+ (1-p)^{\alpha_1-1}}. 
\]

Next, we seek for situations where $h(p)$ is initially increasing and later decreasing. As a first step, we identify situations where 
\[
S^{-}\left(\frac{\partial}{\partial p }h\right)\le 1,  \text{  on the interval } (0,1),
\]
or, equivalently,
\[
S^-(l))= S^{-}(l_1 + l_2 - l_3) \leq 1, \text{  on the interval } (0,1),
\] 
where
\[
l_1(p)=(\alpha_2- \alpha_1) p^{\alpha_2+ \alpha_1-3}+ (\alpha_1- \alpha_2) (1-p)^{\alpha_2+ \alpha_1-3},
\]
\[
l_2(p)=p^{\alpha_2-2}(1-p)^{\alpha_1-2}[(\alpha_2-1)(1-p)+(\alpha_1-1)p]
\]
and
\[
l_3(p)=p^{\alpha_1-2}(1-p)^{\alpha_2-2}[(\alpha_1-1)(1-p)+(\alpha_2-1)p].
\]

On the one hand, it is easy to see that $l_1$ is decreasing if, and only if, 
\begin{equation}\label{l1}
(\alpha_2- \alpha_1)(\alpha_2+ \alpha_1 -3)\leq 0.
\end{equation}

On the other hand, it is easy to see that the function $l_2$ is decreasing if 
\begin{equation}\label{l2-1}
(\alpha_2-2)(\alpha_2-1) < 0,
\end{equation}
and 
\begin{equation}\label{l2-2}
(\alpha_1-2)(\alpha_1-1) > 0. 
\end{equation}

Analogously, $l_3$ is increasing if \eqref{l2-1} and \eqref{l2-2} are satisfied. Therefore, if \eqref{l1}, \eqref{l2-1} and \eqref{l2-2} are satisfied, then $l(p)$ is decreasing and, consequently, $S^{-}( l)\leq 1$, on the interval $(0,1),$ with the sign sequence $(+,-)$ in case of equality. Moreover, it is easy to see that $\lim_{p\rightarrow 0^{+}} l(p) >0$ and $\lim_{p\rightarrow 1^{-}} l(p)< 0$, under \eqref{l1}, \eqref{l2-1} and \eqref{l2-2}. Therefore,  $S^{-}( l) = 1$, on the interval $(0,1)$, with the sign sequence $(+,-)$, that is,  there exists $p^* \in (0,1)$ such that the ratio $f(F^{-1}(p))/g(G^{-1}(p))$ is increasing over $(0,p^*)$ and decreasing over $(p^*,1)$. Figure \ref{dominio} depicts the region where \eqref{l1}, \eqref{l2-1} and \eqref{l2-2} are satisfied and, therefore, the ratio $f(F^{-1}(p))/g(G^{-1}(p))$ is unimodal. 

\begin{figure}[ht]\centering
\includegraphics[width=\linewidth]{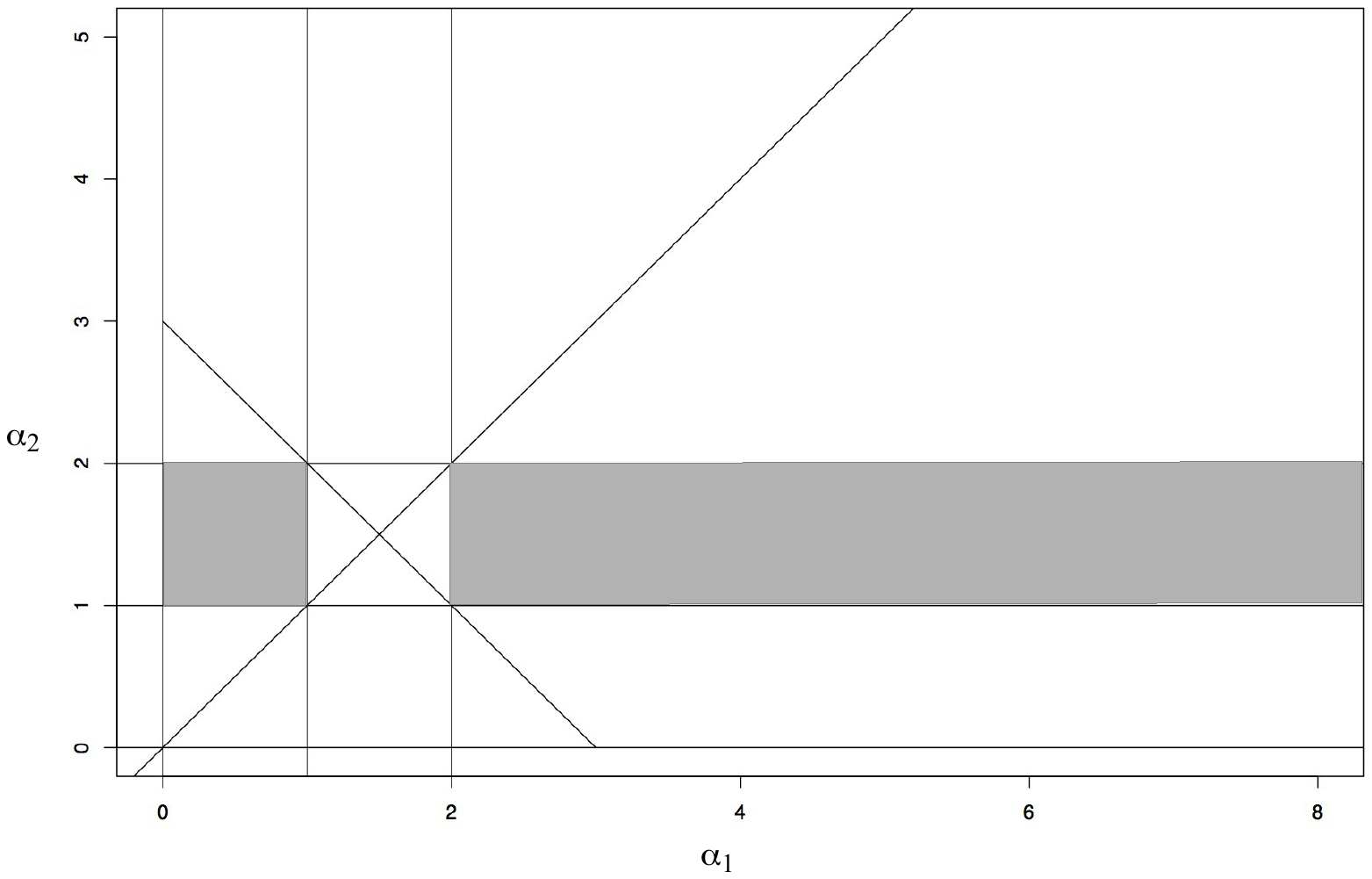}
\caption{\label{dominio}\textit{Domain for $\alpha_1$ and $\alpha_2$ such that the ratio of the quantile density functions is unimodal (grey area).}}
\end{figure} 

In order to conclude, we check the limits of the function $\frac{f(F^{-1}(p))}{g(G^{-1}(p))}(F^{-1}(p)-E[X])-(G^{-1}(p)-E[Y])$ and $\delta$:
\[
\lim_{p\rightarrow 1^-} \frac{f(F^{-1}(p))}{g(G^{-1}(p))}(F^{-1}(p)-E[X])-(G^{-1}(p)-E[Y]) = 
\left\{ 
\begin{array}{ll}
-\eta_2< 0 & \textup{if } \alpha_1<1, \\ 
\eta_2 \left(\frac{\alpha_2}{\alpha_1}-1\right) <0 & \textup{if } \alpha_1> 2.
\end{array}
\right. 
\]

\[
\lim_{p\rightarrow 0^+} \frac{f(F^{-1}(p))}{g(G^{-1}(p))}(F^{-1}(p)-E[X])-(G^{-1}(p)-E[Y])= 
\left\{ 
\begin{array}{ll}
\eta_2 > 0 & \textup{if } \alpha_1<1, \\ 
\eta_2 \left(1-\frac{\alpha_2}{\alpha_1}\right)> 0 & \textup{if } \alpha_1> 2.
\end{array}
\right. 
\]

\[
\lim_{p\rightarrow 0^+} \delta (p) = 
\left\{ 
\begin{array}{ll}
-(\lambda_2-\eta_2) < 0 & \textup{if } \alpha_1<1, \\ 
\frac{\eta_2 \alpha_2}{\eta_1 \alpha_1} (\lambda_1- \eta_1) - (\lambda_2- \eta_2)  & \textup{if } \alpha_1> 2,
\end{array}
\right. 
\]

\[
\lim_{p\rightarrow 1^-} \delta (p) = 
\left\{ 
\begin{array}{ll}
-(\lambda_2+\eta_2) < 0 & \textup{if } \alpha_1<1, \\ 
\frac{\eta_2 \alpha_2}{\eta_1 \alpha_1} (\lambda_1+ \eta_1) - (\lambda_2+ \eta_2)  & \textup{if } \alpha_1> 2.
\end{array}
\right. 
\]

According to the discussion and noticing that $0 \leq \lim_{p\rightarrow 1^-} \delta (p) \leq \lim_{p\rightarrow 0^+} \delta (p)$, whenever $\alpha_1 \in (2,+\infty)$, we can state the following.

Let $X \sim T(\lambda_1, \eta_1, \alpha_1)$ and $Y\sim T(\lambda_2, \eta_2, \alpha_2)$ such that $\alpha_2 \in (1,2)$, then
\begin{enumerate}
	\item The qmit order does not hold.
	\item The dmrl order does not hold.
  \item $X\geq_{\textup{*}}Y$ if one of the following conditions hold
	\begin{enumerate}
		\item $\alpha_1 \in (0,1)$.
		\item $\alpha_1 \in (2,+\infty)$ and $\eta_2 \alpha_2(\lambda_1- \eta_1) < \eta_1 \alpha_1(\lambda_2- \eta_2)$.
	\end{enumerate}
  \item $X\leq_{\textup{*}}Y$ if, and only if, $\alpha_1 \in (2,+\infty)$ and $(\lambda_1+\eta_1) \alpha_2 \eta_2 > (\lambda_2+\eta_2) \alpha_1 \eta_1$.
\end{enumerate}

Case 3 does not characterize the star-shaped order since the sign of $\delta (p^*)$ should be also checked to conclude if $X\geq_{\textup{*}}Y$, for $\alpha_1 < 1$ where $p^*$ denotes the absolute extreme of the ratio $f(F^{-1} (p))/g(G^{-1} (p))$, but we do not have closed expression for $p^*$.

Let us consider the particular case where $X \sim T(4,1,2.5)$ and $Y \sim T(1.5,1,1.5)$. As we can see in Figure \ref{ex-tukey1}, on the left,   the ratio of the densities at the quantiles has just one maximum and the function $\delta$ is always positive (on the right). 
 
\begin{figure}[ht]\centering
\includegraphics[width=\linewidth]{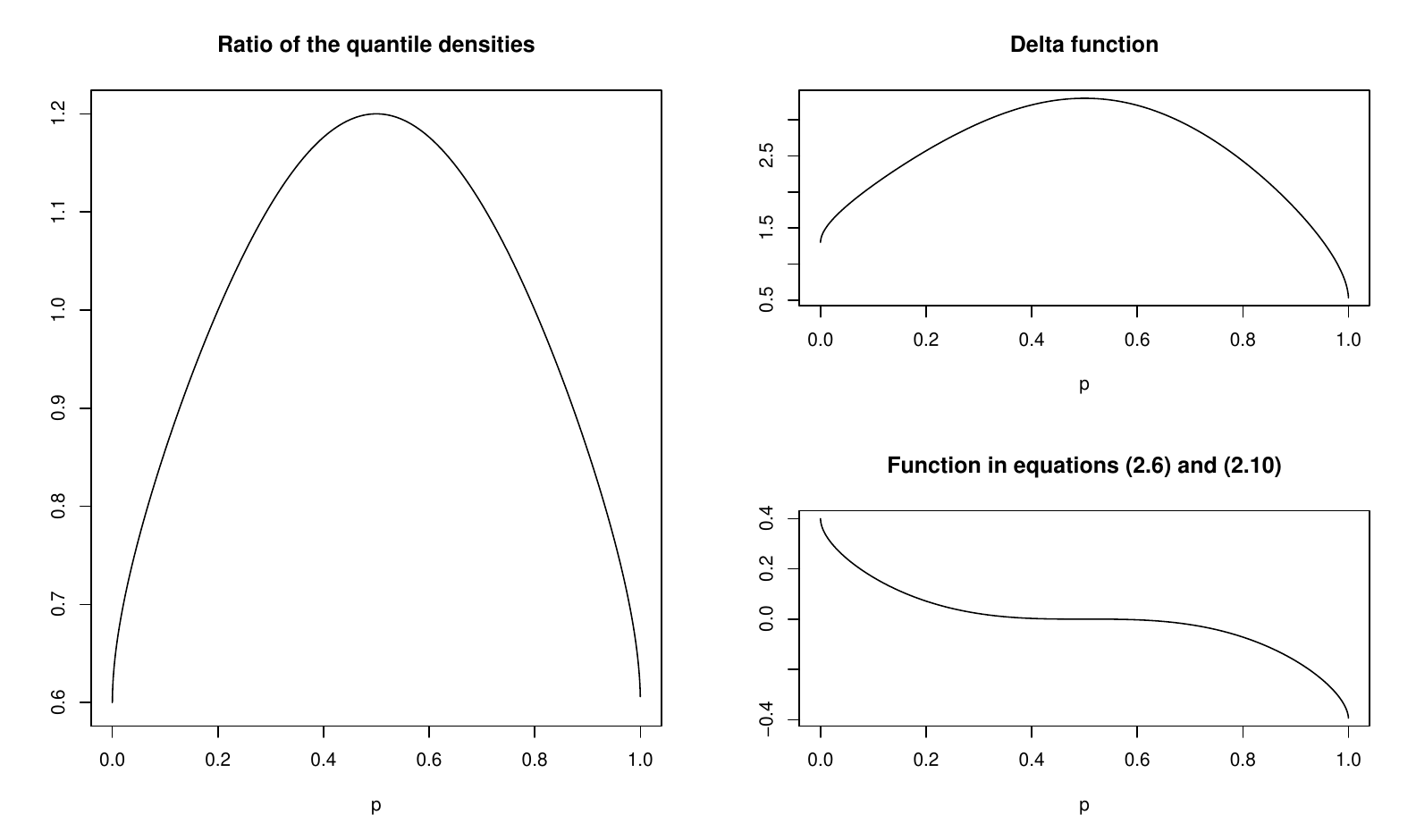} 
\caption{\label{ex-tukey1}\textit{Plots of the ratio of the quantile density functions (left),  the $\delta$ function  (upper right) and the function that appears in \eqref{lim_cond} and \eqref{lim_conddmrl} (lower right), for $X \sim T(4,1,2.5)$ and $Y \sim T(1.5,1,1.5)$.}}
\end{figure} 

From the previous discussion we have that the star-shaped order holds, but neither the qmit order nor the dmrl order hold, as we can see in Figure \ref{ex-tukey2}.

\begin{figure}[ht]\centering
\includegraphics[width=\linewidth]{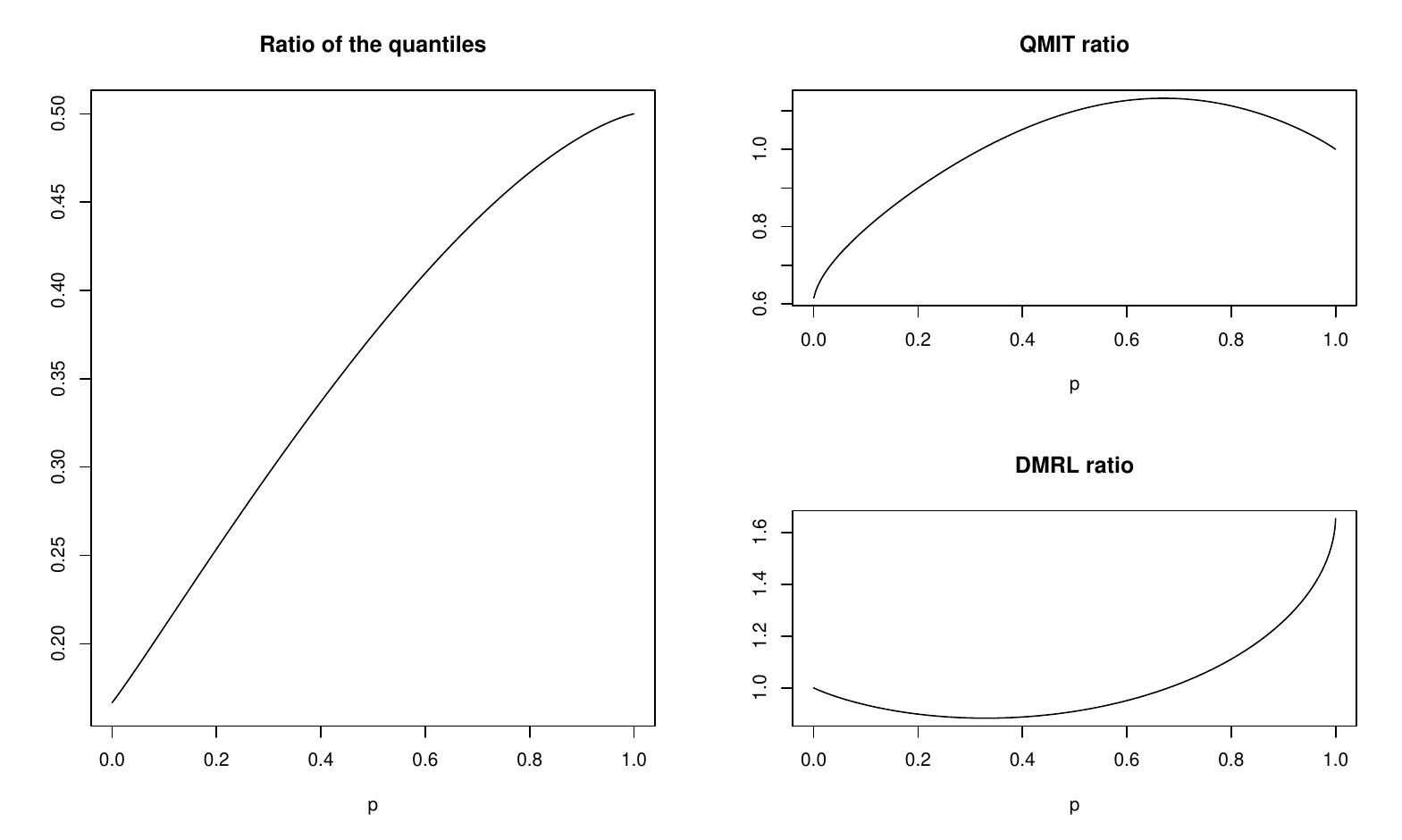} 
\caption{\label{ex-tukey2}\textit{Plots of the ratio of the quantile functions (left), the ratio of the mean inactivity quantile functions (upper right) and the ratio of the mean residual quantile funcions (lower right), for $X \sim T(4,1,2.5)$ and $Y \sim T(1.5,1,1.5)$.}}
\end{figure}

\end{Exam}

\subsection{Some implications in reliability and aging notions} 

One of the main applications of the transform orders is to provide useful characterizations of well-known aging classes, where aging means ``the phenomenon whereby an older system has a shorter remaining lifetime, in some statistical sense than a younger one''(see Bryson and Siddiqui, 1969). Let us recall the definitions, as well as the characterizations, of the aging classes Increasing Failure Rate (IFR), Decreasing Mean Residual Life (DMRL), Increasing Hazard Rate Weighted Average (IHRWA) and Increasing Hazard Rate Average (IFRA). 

\begin{Def}
Let $X$ be a non-negative and continuous random variable with hazard rate $r$ and mean residual life $m$ and $Y$ be an exponentially distributed random variable with scale parameter equal to 1. Then,
\begin{enumerate}

\item $X$ is said to be IFR if $r(t)$ is increasing over the support or, equivalently, if $X\le_c Y$.

\item $X$ is said to be DMRL if $m(t)$ is decreasing over the support or, equivalently, if $X\le_{dmrl} Y$.

\item $X$ is said to be IHRWA [DHRWA] if 
\begin{equation}\label{ifraclas}
\int_{0}^{t} r(x)dW_t (x) \text{  is increasing [decreasing] in } t, \text{ where } W_t(x)=\frac{\int_{0}^{x} F(u)du}{\int_{0}^{t} F(u)du}
\end{equation}
or, equivalently, if $X\le_{qmit} [\ge_{qmit}]Y$.

\item $X$ is said to be IFRA [DFRA] if the hazard rate on average function
\[
 \frac{\int_{0}^{t} r(x)dx}{t} \text{  is increasing [decreasing] in } t\geq 0,
\]
or, equivalently, if $X\le_{\star} [\ge_{\star}] Y$.

\end{enumerate} 
\end{Def}

We refer the reader to (Shaked et. al 2007 and Kochar et. al 1987) to consult the proofs of the given characterizations in the case of the aging classes IFR, DMRL and IFRA. Regarding the class IHRWA, this is a recent aging notion which appears as an intermediate class between IFR and IFRA (see Arriaza et. al, (2017) for definition and characterization with the qmit order). Indeed, it is expressed as a property of the hazard rate function $r(x)$, see \eqref{ifraclas}. It represents the monotonicity of the weighted average of $r(x)$ and can be considered as a property weaker than monotonicity and stronger than monotonicity in average of $r(x)$.

From \eqref{cuadro}, it is clear that the following implications hold among the different aging notions: 
\begin{equation*}
\begin{tabular}{c c c c c }
  \text{IFR} & $\Rightarrow$ & \text{IHRWA} & $\Rightarrow$ &  \text{IFRA} \\
 & \rotatebox{45}{$\Downarrow$} & &  &   \\
  &  & \text{DMRL} &  &   
\end{tabular}
\end{equation*}

Frequently, survival and failure times are modeled by monotone reliability functions as in the previous definitions. However, this is inappropriate in some cases, for example, where mortality caused by a disease reaches a peak after a period and then  decreases slowly. In such cases, a unimodality model will be more desirable. Lai and Xie (2006) give a further discussion on the topic of non-monotone reliability functions with several examples. In particular, let us recall the definition of the property of having a bathtub [upside down bathtub] hazard rate function. 

\begin{Def}
Let $X$ be a non-negative and continuous random variable. It is said that $X$ has a bathtub (BT) [upside down bathtub (UBT)] hazard rate if the hazard rate function starts decreasing [increasing] and ends increasing [decreasing]. Analogously, it is defined a BT [UBT] mean residual life function.
\end{Def}

For a non-negative and continuous random variable $X$ and an exponentially distributed random variable $Y$ with scale parameter equal to 1, we get that
\[
\frac{f(F^{-1}(p))}{g(G^{-1}(p))} = r(F^{-1}(p)).
\]
Consequently, the behavior of the hazard rate $r(t)$ in $t$ is the same as the one of the ratio of the quantile density functions in $p$. That means, the results given in Section \ref{S2} provide new characterizations of some aging classes for random variables with BT or UBT hazard rates. In particular, the following corollaries for the DMRL, IHRWA and IFRA aging classes follow from Corollaries \ref{Th2.5} and \ref{Th2.3} and Theorem \ref{Th2.2}, respectively.
 
\begin{Cor}\label{Cor1}
Let $X$ be a non-negative random variable with interval support, mean residual life $m$ and BT [UBT] hazard rate $r$. Then, 
\begin{itemize}
	\item $m(t)$ is decreasing [increasing] if, and only if, $\lim_{p \rightarrow 0^+} r(F^{-1}(p)) \le [ \ge ] 1/E[X].$
	\item Otherwise, $m(t)$ is UBT [BT].
\end{itemize}
\end{Cor}

\begin{Cor}
Let $X$ be a non-negative random variable with interval support and UBT [BT] hazard rate $r$. Then 
\begin{itemize}
	\item $X$ is IHRWA [DHRWA] if, and only if,
 \[
 \lim_{p \rightarrow 1^-} \left(r(F^{-1}(p)(F^{-1}(p)-E[X])+\log(1-p) +1\right) \geq [\leq] 0. 
 \] 
	\item Otherwise, the function $\int_{0}^{t} r(x)dW_t (x)$ is UBT [BT].
\end{itemize}
\end{Cor}

\begin{Cor}
Let $X$ be a non-negative random variable with interval support and UBT [BT] hazard rate $r$, where the maximum (minimum) is attained at $F^{-1}(p^*)$. Then, 
\begin{itemize}
	\item $X$ is IFRA  [DFRA] if, and only if, 
	$$\lim_{p \rightarrow 0^+} \left(F^{-1}(p)r(F^{-1}(p))+\log(1-p)\right) \geq [\leq] 0$$
	\noindent and 
	$$\lim_{p \rightarrow 1^-}  \left(F^{-1}(p)r(F^{-1}(p))+\log(1-p)\right)  \geq [\leq] 0. $$
        \item $X$ is DFRA [IFRA] if, and only if, 
	$$ \left(F^{-1}(p^*)r(F^{-1}(p^*))+\log(1-p^*)\right)  \leq [\geq] 0.$$
\end{itemize}
\end{Cor}
  
As long as Corollary \ref{Cor1} has been given in Gupta and Akman (1995a) and (1995b) (see also Lai and Xie, 2006, p. 118, for a different approach), the remaining corollaries can not be found in the literature, as far as we know. These results show that the aging of a random variable can be monotone in some sense when the hazard rate is not monotone. 

Nair et. al (2013) provide a detailed discussion on parametric families with closed expressions for the quantile functions. They also analyze the shape of their hazard rate functions. In particular, they show that the hazard rate of a Govindarajulu distributed random variable is monotone for $\beta \leq 1$ and BT for $\beta >1$. So, let us apply the previous corollaries to check if the aging of a Govindarajulu distributed random variable is monotone, in some other sense, when $\beta>1$, i.e., whenever the hazard rate function is non-monotone.

\begin{figure}[ht]\centering
\includegraphics[scale=0.45]{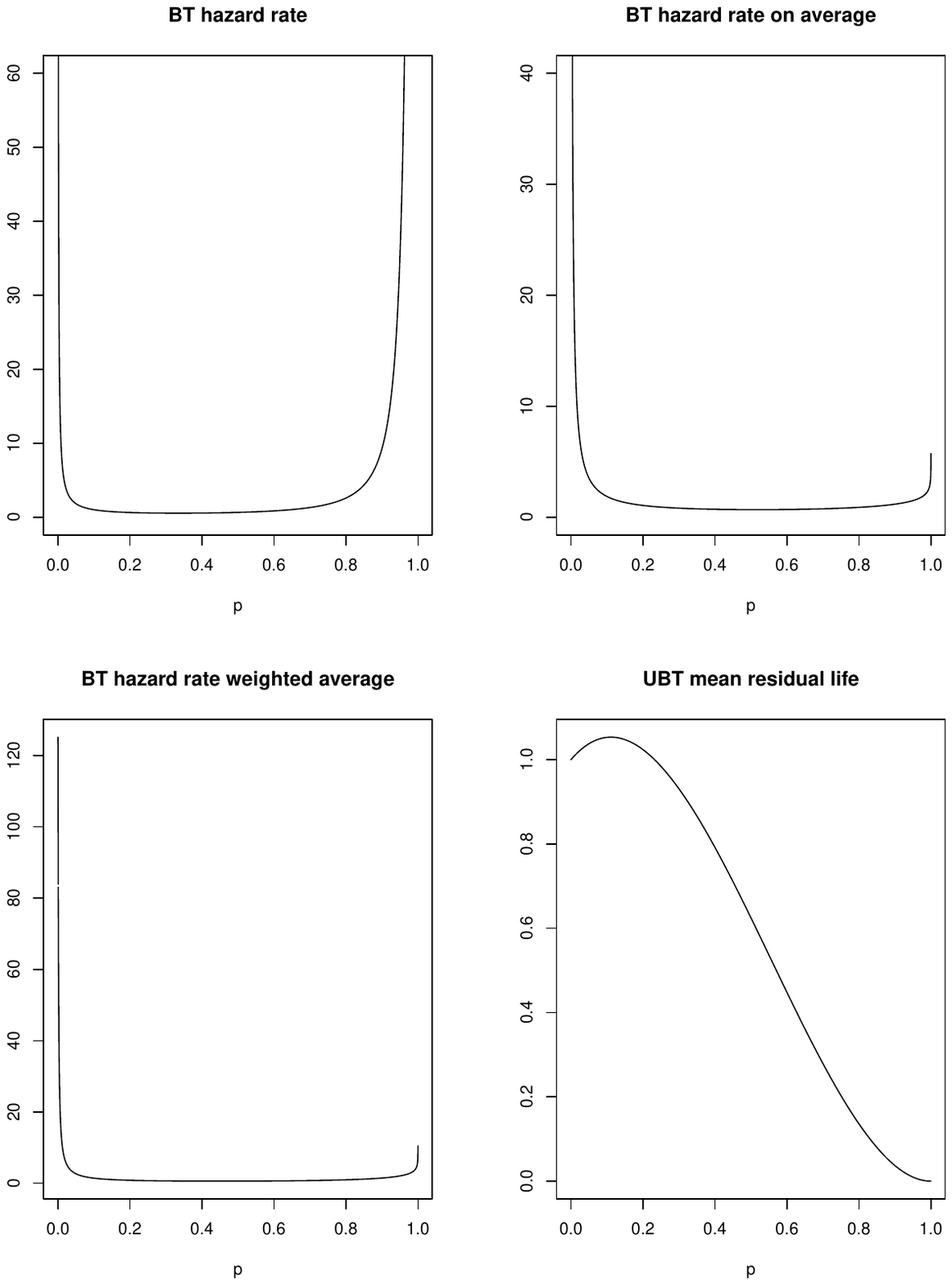} 
\caption{\label{aging_ejemplo}\textit{Plots of the hazard rate function (upper left plot), the hazard rate on average function (upper right plot), the function given in \eqref{ifraclas} (lower left plot) and the hazard rate weighted on average function (lower right plot), for $X \sim G(0,2,2)$.}}
\end{figure} 

\begin{Exam} Let us consider a Govindarajulu distributed random variable $X$ and an exponentially distributed random variable $Y$, that is, $X \sim G(\theta, \sigma, \beta)$ and $Y\sim Exp(1)$, with quantile functions
\[
F^{-1}(p)= \theta + \sigma ((\beta+1)p^{\beta}-\beta p^{\beta+1}) \text{ , for all } p\in (0,1)
\]
and 
\[
G^{-1}(p)=-\log(1-p),
\]
\noindent respectively, where $\theta \geq 0$, $\sigma>0$ and $\beta>0$. Let us denote by $r$ and $m$ the hazard rate and the mean residual life function of $X$, respectively. According to the previous discussion, we assume $\beta > 1$ and, consequently, $r(F^{-1}(p))$ is BT with mode $p^*=(\beta -1)/(\beta +1)$ (see Nair et. al, 2013). In order to analyze the aging of $X$ in such cases, we proceed to check the corresponding limit conditions appearing in the previous corollaries. It is easy to see that
\[
\lim_{p \rightarrow 0^+} r(F^{-1}(p)) = +\infty > \frac{1}{ E[X]}, 
\]
therefore, we conclude that $m(t)$ is UBT. Furthermore, 
\[
 \lim_{p \rightarrow 1^-} (r(F^{-1}(p))(F^{-1}(p)-E[X])+\log(1-p)+1)= +\infty,
\]
then, the function $\int_{0}^{t} r(x)dW_t (x)$ is BT. Finally, we have that 
\[
\lim_{p \rightarrow 0^+}  (F^{-1}(p)r(F^{-1}(p))+\log(1-p))= 0 \text{ and }
\lim_{p \rightarrow 1^-} (F^{-1}(p)r(F^{-1}(p))+\log(1-p))=+\infty \geq 0,
\]
that means, $X$ is not DFRA. Now, $X$ is IFRA if, and only if, 
\[
(F^{-1}(p^*)r(F^{-1}(p^*))+\log(1-p^*)) \geq 0.
\]
To sum up, for $X \sim G(\theta, \sigma, \beta)$, we have that 
\begin{itemize}
\item $r(t)$ is monotone for all $\beta \leq 1$.
\item $r(t)$ is BT, $m(t)$ is UBT, $\int_{0}^{t} r(x)dW_t (x)$ is UBT, for all $\beta \geq 1$.
\item $X$ is not DFRA, for all $\beta \geq 1$.
\item $X$ is IFRA if, and only if, $(F^{-1}(p^*)r(F^{-1}(p^*))+\log(1-p^*)) \geq 0$, for all $\beta \geq 1$.
\end{itemize}

Figure \ref{aging_ejemplo} shows a particular case for $X \sim G(0, 2, 2)$. As for the hazard rate weighted average plot is concerned, we have plotted the function
\[
\frac{\int_{0}^{p}\frac{q}{1-q} dq }{\int_{0}^{p}\frac{q}{f(F^{-1}(p))} dq},
\]
\noindent which has the same behavior than the  hazard rate weighted average function.
\end{Exam}

\section{Conclusions}

In the present paper, we have introduced new characterizations for the star-shaped order, the ps order, the qmit order and the dmrl order when the ratio of the quantile density functions has $n\geq 1$ finite relative extremes. These results provide a new perspective when we are interested in checking if any of these orders hold, even though the convex transform order is not satisfied. We have illustrated the utility of the proposed characterizations with an application of these results to compare two Tukey generalized distributed random variables according to their parameters. 

Furthermore, we want to point out that the unimodality of the quantile density ratio is not a strong assumption. In fact, there exist a vast number of probability distributions in literature which satisfy this condition for some values of their parameters. Indeed, we have provided several examples where this condition holds.

Finally, as a direct consequence of our main results, we have provided new characterizations for the DMRL, the IFRWA and the IFRA aging notions when the IFR is not satisfied. 

\section*{Acknowledgements} The authors want to acknowledge the comments by an anonymous referee that have greatly improved the presentation and the contents of this paper. We also want to thank Joseph Hooly for his help to improve this paper language-wise. F\'{e}lix Belzunce and Carolina Mart\'{i}nez-Riquelme want to acknowledge the support received by the Ministerio de Econom\'{i}a, Industria y Competitividad under grant MTM2016-79943-P (AEI/FEDER, UE) and Antonio Arriaza acknowledges the support received by the Ministerio de Econom\'{i}a y Competitividad under grant MTM2017-89577-P.


\begin{thebibliography}{99}

\bibitem{} Alzaid, A.A. and Al-Osh, M. (1989). Ordering probability distributions by tail behavior. \textsl{Statistics and Probability Letters},  \textbf{8}, 185--188. 

\bibitem{} Arnold, B.C. and Sarabia, J.M. (2018). \textsl{Majorization and the Lorenz Order with Applications in Applied Mathematics and Economics.} Springer, Cham.

\bibitem{} Arriaza, A., Sordo, M.A. and Su\'{a}rez-Llorens, A. (2017). Comparing residual lives and inactivity times by transform stochastic orders. \textsl{IEEE Transactions on Reliability}, \textbf{66}, 366--372.

\bibitem{} Belzunce. F. and Mart\'{i}nez-Riquelme, C. (2017). On sufficient conditions for the comparison of some quantile-based measures. \textsl{Communications in Statistics-Theory and Methods}, \textbf{46}, 6512--6527.

\bibitem{} Belzunce, F., Mart\'{i}nez-Riquelme, C. and Mulero, J. (2015). \textsl{An Introduction to Stochastic Orders}. Elsevier/Academic Press, Amsterdam.

\bibitem{} Belzunce. F., Mart\'{i}nez-Riquelme, C., Ruiz, J.M. and Sordo, M.A. (2016). On sufficient conditions for the comparison in the excess wealth order and spacings. \textsl{Journal of  Applied Probability}, \textbf{53}, 33--46.

\bibitem{} Belzunce. F., Mart\'{i}nez-Riquelme, C., Ruiz, J.M. and Sordo, M.A. (2017). On the comparison of relative spacings with applications. \textsl{Methodology and Computing in Applied Probability}, \textbf{19}, 357--376.

\bibitem{} Belzunce. F., Pinar, J.F., Ruiz, J.M. and Sordo, M.A. (2012). Comparison of risks based on the expected proportional shortfall. \textsl{Insurance: Mathematics and Economics}, \textbf{51}, 292--302.

\bibitem{} Bjerkedal, T. (1960). Acquisition of resistance in guinea pigs infected with different doses of virulent tubercle bacilli. \textsl{American Journal of Hygiene}, \textbf{72}, 130--148.

\bibitem{} Bryson, M.C. and Siddiqui, M.M. (1969). Some criteria for aging. \textsl{Journal of American Statistical Association}, \textbf{64}, 1472--1483. 

\bibitem{} Gupta, R.C. and Akman, H.O. (1995a). Mean residual life function for certain types of non-monotonic ageing. \textsl{Communications in Statistics - Theory and Methods}, \textbf{11}, 219--225. 

\bibitem{} Gupta, R.C. and Akman, H.O. (1995ab). Erratum: Mean residual life function for certain types of non-monotonic ageing. \textsl{Communications in Statistics - Theory and Methods}, \textbf{11}, 561--562. 

\bibitem{}  Kayid, M.; Izadkhah, S. and Alfifi, A. (2018) Mean inactivity time ordering: A quantile approach. \textsl{Mathematical Problems in Engineering}, DOI: 10.1155/2018/8756969. 

\bibitem{} Kochar, S.C. and Wiens, D.P. (1987). Partial orderings of life distributions with respect to their aging properties. Naval Reserach Logistics 34 823--829.

\bibitem{} Lai, C.D. and Xie, M. (2006). \textsl{Stochastic Ageing and Dependence for Reliability}. Springer, New York.

\bibitem{} Marshall, A.W. and Olkin, I. (1979): \textsl{Inequalities: theory of majorization and its applications}. Mathematics in Science and Engineering, 143. Academic Press,  New York-London.

\bibitem{} M\"{u}ller, A. and Stoyan, D. (2002). \textsl{Comparison methods for stochastic models and risks}. Wiley Series in Probability and Statistics. John Wiley \& Sons, Ltd., Chichester.

\bibitem{} Nair, N.U., Sankaran, P. G. and Balakrishnan, N. (2013). \textsl{Quantile-based reliability analysis}. Statistics for Industry and Technology. Birkh\"{a}user/Springer, New York.

\bibitem{} Parzen, M. (1979). Nonparametric statistical data modelling. \textsl{Journal of the American Statistical Association}, 74, 105--121. 

\bibitem{} Shaked, M. and Shanthikumar, G.J. (2007). \textsl{Stochastic Orders}. Springer Series in Statistics. Springer, New York.

\bibitem{} van Zwet, W. R. (1964). \textsl{Convex Transformations of Random Variables}. Mathematical Centre Tracts No. 7. Mathematical Centre, Amsterdam.

\bibitem{} Yu, B.P., Masoro, E.J., Murata, I., Bertrand, H.A. and Lynd, F.T. (1982). Lifespan study of SPF Fisher 344 male rats fed \textit{ad libitum} or restricted diets: Longevity, growth, lean body mass and disease. \textsl{Journal of Gerontology}, \textbf{37}, 130--141. 
\end{thebibliography}
\end{document}